\documentclass[11pt]{article}
\usepackage[T1]{fontenc}
\usepackage[utf8]{inputenc}
\usepackage{lmodern}
\usepackage{microtype}
\usepackage[a4paper,margin=30mm]{geometry}
\usepackage{amsmath,amssymb,amsthm,mathtools}
\usepackage{array}
\usepackage{booktabs}
\usepackage{enumitem}
\usepackage[hidelinks]{hyperref}
\usepackage[nameinlink,noabbrev]{cleveref}

\newtheorem{theorem}{Theorem}[section]
\newtheorem{proposition}[theorem]{Proposition}
\newtheorem{lemma}[theorem]{Lemma}
\newtheorem{corollary}[theorem]{Corollary}

\theoremstyle{definition}
\newtheorem{definition}[theorem]{Definition}

\theoremstyle{remark}

\DeclareMathOperator{\Aut}{Aut}

\title{An infinite family of intransitive \\ directed strongly regular graphs\\
       with rank 6 Weisfeiler--Leman closure}
\author{\v{S}tefan Gy\"urki\footnote{e-mail address: {\tt stefan.gyurki@stuba.sk}}\\[0.1cm]
{\small Slovak University of Technology, Bratislava, Slovak Republic} }
\date{\today}

\begin{document}

\maketitle

\begin{abstract}
We construct an infinite family of directed strongly regular graphs (DSRGs)
with intransitive full automorphism groups whose Weisfeiler--Leman closure
are association schemes of rank 6.  This is the smallest
possible rank for an association scheme admitting a proper DSRG merging.
We also exhibit rigid sporadic DSRGs with the same closure property, showing
that combinatorial regularity, group-theoretic symmetry, and
Weisfeiler--Leman regularity capture fundamentally different aspects of
graph structure.
\end{abstract}



\section{Introduction}
\label{intro}
Symmetries of the graphs can be quantified in many different manners: number of automorphisms, combinatorial regularity properties, rank of the Weisfeiler-Leman closure (WL-closure), etc. 
In this paper we show an example of an infinite family of digraphs that are strongly regular, having the smallest possible rank of their WL-closure, but not vertex-transitive. 
We also exhibit rigid sporadic DSRGs with the same closure property, showing
that combinatorial regularity, group-theoretic symmetry, and
Weisfeiler--Leman regularity capture fundamentally different aspects of
graph structure.

The present work, surprisingly, arose as a byproduct of a systematic search for vertex-transitive directed strongly regular graphs (DSRGs) \cite{bg, gs}.  Our
original objective was to find vertex-transitive DSRGs among unions of relations in association schemes.  
The search also produced a DSRG
on $30$ vertices whose full automorphism group is intransitive, although its 
WL-closure is homogeneous.  This is the smallest example with these properties in our exhaustive search. Even more remarkably, its
WL-closure has rank 6, which is the smallest rank possible for a noncommutative association scheme admitting a proper DSRG as a merging
\cite{kmmz}.  
The next smallest example of an intransitive DSRG having homogeneous WL-closure with rank 6 appear on 38 vertices.  
The similarity between these two examples suggested that the phenomenon is
not accidental, thus we started to investigate whether is it possible to construct an infinite family of DSRGs that are not vertex-transitive, but having the 
smallest possible rank for their WL-closure. 
A careful analysis revealed that both graphs are instances
of a uniform construction based on doubly regular tournaments. As a consequence, based on the Szekeres tournaments we prove the main result of the paper:

\begin{theorem}
\label{thm:main}
For every prime $p\equiv5\pmod8$, $p\geq13$, there exists an intransitive
directed strongly regular graph with parameters
$(4p+2,2p,p,p-1,p)$
whose Weisfeiler--Leman closure is a homogeneous noncommutative association
scheme of rank six.  Consequently, there exist infinitely many intransitive
DSRGs having homogeneous noncommutative WL-closure of the smallest possible
rank.
\end{theorem}
The infinitude follows from \emph{Dirichlet's theorem} on primes in arithmetic
progressions.

The paper is organized as follows.  In Section~\ref{prelim} we
recall the necessary facts about DSRGs, coherent configurations,
Weisfeiler--Leman closure, and doubly regular tournaments.
Section~\ref{sec:main} contains the doubling construction, the
proof that its WL-closure has rank six, the analysis of its automorphisms, and
the application to the Szekeres tournaments.  
We also briefly discuss the smallest rigid DSRGs obtained by the doubling construction. 

\section{Preliminaries}
\label{prelim}

Throughout the paper, all graphs, digraphs, groups and configurations are
finite.  For a positive integer $n$, we denote by $I=I_n$ and $J=J_n$ the
identity matrix and the all-one matrix of order $n$, respectively.  Matrix
transpose is denoted by $A^{\mathsf T}$ and entrywise (Schur-Hadamard) 
multiplication by $A\circ B$.  
If $X$ is a set, then $\operatorname{Sym}(X)$ denotes its full
symmetric group.

\subsection{Directed strongly regular graphs}

A digraph $\Gamma=(V,\mathcal A)$ consists of a finite vertex set $V$ and a
set of arcs
$ \mathcal A\subseteq V\times V$.
All digraphs considered here are loopless, so $(x,x)\notin\mathcal A$ for
every $x\in V$.  Both arcs $(x,y)$ and $(y,x)$ are allowed to occur.  In this
case, $x$ and $y$ are joined by an \emph{undirected edge}.  The adjacency
matrix $A=A(\Gamma)$ is the $(0,1)$-matrix indexed by $V$ for which
$A_{xy}=1\iff (x,y)\in\mathcal A$.
The \emph{reverse} digraph $\Gamma^{\mathsf T}$ has adjacency matrix
$A^{\mathsf T}$.

Directed strongly regular graphs were introduced by Duval
\cite{du} as a directed analogue of strongly regular graphs.  

\begin{definition}
A loopless digraph $\Gamma$ is a \emph{directed strongly regular graph} 
(DSRG, for short) with parameters $(v,k,t,\lambda,\mu)$ if $|V|=v$ 
and its adjacency matrix satisfies
\[ AJ=JA=kJ\quad \text{ and }  A^2=tI+\lambda A+\mu(J-I-A).\]
\end{definition}

The first equality means that every vertex has both in-degree
and out-degree $k$.  The diagonal entries in the second equality show
that every vertex has exactly $t$ neighbours joined to it in both
directions.  For two distinct vertices $x$ and $y$, the corresponding
off-diagonal entry states that the number of directed walks of length
two from $x$ to $y$ is $\lambda$ if $x\to y$, and $\mu$ otherwise.
A DSRG is called \emph{proper} (or also \emph{genuine}) if $0<t<k$.
The cases $t=k$ and $t=0$ include undirected strongly regular graphs and
doubly regular tournaments, respectively.  We follow the parameter order
$(v,k,t,\lambda,\mu)$ used in \cite{du}.
 Since their introduction, DSRGs have been studied by spectral, group-theoretic, design-theoretic and algebraic methods; see, for example,
\cite{du,fkm,fkp,gh,gp,hs,jo,kmmz}.

The automorphism group $\Aut(\Gamma)$ consists of all permutations of $V$
preserving the arc set.  The digraph is \emph{vertex-transitive} if
$\Aut(\Gamma)$ acts transitively on $V$, \emph{intransitive} otherwise, and
\emph{rigid} if $\Aut(\Gamma)=1$.  We use the word rigid rather than
asymmetric, since ``asymmetric digraph'' is also used in the literature for a
digraph having no pair of oppositely directed arcs.

\subsection{Coherent configurations and association schemes}

Let $\Omega$ be a finite set.  For a binary relation
$R\subseteq\Omega\times\Omega$, its \emph{transpose} is
$ R^{\mathsf T}=\{(y,x):(x,y)\in R\}$.
The \emph{diagonal relation} is denoted by $1_{\Omega}=\{(x,x):x\in\Omega\}$.

\begin{definition}
A \emph{coherent configuration} on the set $\Omega$ is a pair
$\mathcal X=(\Omega,\mathcal R)$, where
$ \mathcal R=\{R_0,R_1,\ldots,R_{r-1}\}$
is a partition of $\Omega\times\Omega$ satisfying the following conditions:
\begin{enumerate}
 \item the diagonal relation $1_{\Omega}$ is a union of members of $\mathcal R$;
 \item for every relation $R_i\in\mathcal R$, its transpose $R_i^{\mathsf T}$ also belongs
       to $\mathcal R$;
 \item for all $i,j,k$ there is a nonnegative integer $p_{ij}^{k}$ such that,
       for every $(x,y)\in R_k$, the number
       $|\{z\in\Omega:(x,z)\in R_i,\ (z,y)\in R_j\}|$ is equal to $p_{ij}^{k}$.
\end{enumerate}
The relations $R_i$ are called the \emph{basis relations}, the numbers
$p_{ij}^{k}$ are the \emph{intersection numbers} (or \emph{structure constants}), and
$r=|\mathcal R|$ is the \emph{rank} of $\mathcal X$.
\end{definition}

Let $A_i$ denote the adjacency matrix of $R_i$.  The third condition is
equivalent to
\[ A_iA_j=\sum_{k=0}^{r-1}p_{ij}^{k}A_k. \]
The vector space $\mathbb C\mathcal X   
=\operatorname{span}_{\mathbb C}\{A_0,A_1,\ldots,A_{r-1}\}$
is the \emph{adjacency algebra} or \emph{coherent algebra} of
$\mathcal X$.  It contains $I$ and $J$ and is closed under ordinary matrix
multiplication, transpose and entrywise (Schur-Hadamard) multiplication.

The diagonal basis relations are called the \emph{fibres} of the coherent
configuration.  A coherent configuration is \emph{homogeneous} if
$1_{\Omega}$ itself is a basis relation, that is, if it has exactly one
fibre.  In this paper an \emph{association scheme} means a homogeneous
coherent configuration.  An association scheme is \emph{commutative} if its
adjacency algebra is commutative, or equivalently if
$ p_{ij}^{k}=p_{ji}^{k}$
for all $i,j,k$; otherwise it is \emph{noncommutative}.

Let $G\leq\operatorname{Sym}(\Omega)$.  The orbits of the componentwise
action of $G$ on $\Omega\times\Omega$ are called the \emph{orbitals} of
$G$ and form a coherent configuration, denoted by $\operatorname{Inv}(G)$.
A coherent configuration is \emph{Schurian} if it is equal to
$\operatorname{Inv}(G)$ for some permutation group $G$.  A Schurian coherent
configuration is homogeneous precisely when the corresponding group is
transitive.  Thus a homogeneous coherent configuration whose full
automorphism group is intransitive is necessarily \emph{non-Schurian}.

An automorphism of $\mathcal X=(\Omega,\mathcal R)$ is a permutation of
$\Omega$ preserving every basis relation.  Its automorphism group is denoted
by $\Aut(\mathcal X)$.  Equivalently,
\[ \Aut(\mathcal X) =\{g\in\operatorname{Sym}(\Omega):P_gA_i=A_iP_g
     \text{ for every }i\},\]
where $P_g$ is the permutation matrix of $g$.

Let $\mathcal X=(\Omega,\{R_0,\ldots,R_{r-1}\})$ be a coherent
configuration.  A union $ R=\cup_{i\in S}R_i$
of basis relations is called a \emph{merging} of relations; its adjacency
matrix is $\sum_{i\in S}A_i$.  We say that a graph or digraph is obtained as
a merging in $\mathcal X$ if its arc relation is such a union.

A \emph{fusion} of $\mathcal X$ is a partition of the basis relations into
classes such that the corresponding unions again form the basis relations of
a coherent configuration.  Thus every fusion is a coherent configuration,
whereas an individual merging need not itself determine a fusion. 

We shall use the following known restriction on the association scheme
approach to proper DSRGs.

\begin{proposition}[{\cite{fkp,kmmz}}]
\label{ranksix}
If a proper directed strongly regular graph is obtained as a nontrivial
merging in an association scheme, then that association scheme is
noncommutative and has rank at least six.
\end{proposition}

This lower bound explains the importance of the rank-six closures occurring
in the present paper: they attain the smallest rank compatible with such a
merging.

\subsection{Weisfeiler--Leman closure}

The two-dimensional Weisfeiler--Leman algorithm starts with a coloring of the
ordered pairs of vertices that distinguishes the diagonal, arcs and non-arcs,
and repeatedly refines the coloring according to the numbers of intermediate
vertices of every ordered pair of colors.  The stable partition is a coherent
configuration; see \cite{hi,wl}.  Algebraically, it
has the following equivalent description.

\begin{definition}
Let $\Gamma$ be a graph or digraph with adjacency matrix $A$.  Its
\emph{Weisfeiler--Leman closure}, (WL-closure, for short) 
denoted by $\operatorname{WL}(\Gamma)$, is
the smallest coherent configuration whose adjacency algebra contains $A$.
Equivalently, it is the smallest coherent configuration in which the arc
relation of $\Gamma$ is a union of basis relations.
\end{definition}

The construction is canonical, and hence
$ \Aut(\Gamma)=\Aut(\operatorname{WL}(\Gamma))$.
In particular, every automorphism of $\Gamma$ preserves every basis relation
of its WL-closure.  The closure need not be homogeneous.  If it is
homogeneous, then it is an association scheme.  In this paper the
\emph{WL-rank} of $\Gamma$ means the rank of
$\operatorname{WL}(\Gamma)$.  This terminology avoids confusing the rank of
the coherent configuration with the ordinary linear-algebraic rank of the
adjacency matrix.

If $\Gamma$ is vertex-transitive, then its WL-closure is homogeneous, since
the diagonal color classes are invariant under $\Aut(\Gamma)$.  The converse
is false and is central to this work: our graphs have homogeneous WL-closure
although their automorphism groups may be intransitive or even trivial.

\subsection{Tournaments and doubly regular tournaments}

A \emph{tournament} $T$ is an orientation of a complete graph.  Its adjacency
matrix $B$ satisfies  $B+B^{\mathsf T}=J-I$.
The \emph{reverse tournament} $T^{\mathsf T}$ has all arcs reversed and adjacency
matrix $B^{\mathsf T}$.  A tournament is \emph{regular} if every vertex has
the same out-degree.  A regular tournament necessarily has odd order
$2d+1$ and out-degree $d$.

For a tournament $T$, an \emph{anti-automorphism} is an isomorphism
$T\to T^{\mathsf T}$.  We use the \emph{extended automorphism group}
$ \Aut^{\pm}(T) =\{g\in\operatorname{Sym}(V(T)):
       T^g=T\text{ or }T^g=T^{\mathsf T}\}$.
The subgroup $\Aut(T)$ is normal in $\Aut^{\pm}(T)$ and has index at most
two.  The tournament is \emph{self-reversed} if $T\cong T^{\mathsf T}$.

A tournament $T$ is called \emph{doubly regular}
if it is regular and every two
distinct vertices have the same number of common out-neighbours.
If this common number is $q$, then a doubly regular tournament has order
$4q+3$ and out-degree $2q+1$.  If $B$ is its adjacency matrix, then
\begin{align*}
 BJ=JB&=(2q+1)J, \\
 B+B^{\mathsf T}&=J-I, \\
 BB^{\mathsf T}=B^{\mathsf T}B
   &=(q+1)I+qJ, \\
 B^2&=qB+(q+1)B^{\mathsf T}. 
\end{align*}
Conversely, these identities characterize a doubly regular tournament.  In
the terminology of DSRGs, a doubly regular tournament is a 
DSRG with parameters $(4q+3,2q+1,0,q,q+1)$.

\section{The construction and the infinite family}
\label{sec:main}

In this section we present the main construction.  Starting with an arbitrary
doubly regular tournament, we obtain a directed strongly regular graph whose
WL-closure is a noncommutative association scheme of rank six.
We then prove that the construction preserves intransitivity and apply it to
the Szekeres tournaments.  This yields the infinite family announced in
Theorem~\ref{thm:main}.

\subsection{The doubling construction}

Let $T$ be a doubly regular tournament of order $4q+3$, where $q\geq0$, and
let $B$ be its adjacency matrix.  Put
$n=4q+3$, and  $m=2q+1=(n-1)/2$.
Thus $m$ is the out-valency of $T$.  
Define the $(2n)\times(2n)$ matrix
\[ A=A(B)=
 \begin{pmatrix}
  B&B^{\mathsf T}\\
  B&B^{\mathsf T}
 \end{pmatrix},
\]
and let $\Gamma(T)$ denote the corresponding digraph.  This matrix belongs to
a block construction considered more generally in \cite[Lemma 4.2]{ag}.  

\begin{lemma}(\cite[Lemma 4.2]{ag})
\label{thm:doubling}
Let $T$ be a doubly regular tournament of order $4q+3$ and out-valency
$m=2q+1$.  Then $\Gamma(T)$ is a directed strongly regular graph with
parameters
\[ (v,k,t,\lambda,\mu)=(8q+6,4q+2,2q+1,2q,2q+1)=(4m+2,2m,m,m-1,m). \]
Equivalently, its adjacency matrix satisfies
$ A^2+A=mJ_{2n}$.
\end{lemma}

The vertex set of $\Gamma(T)$ will be written as $ V(T)\times\{0,1\}$.
The two sets $V(T)\times\{0\}$ and $V(T)\times\{1\}$ are called the
\emph{natural layers}.  Notice that $(x,0)$ and $(x,1)$ have the same
out-neighbourhood, but they are distinguished by their in-neighbourhoods.

Here we will concentrate on the coherent
closure and automorphism group of $\Gamma(T)$.

\subsection{The rank-six Weisfeiler--Leman closure}

Consider the following six $(0,1)$-matrices of order $2n$:
\[
\begin{gathered}
R_0=\begin{pmatrix}I&0\\0&I\end{pmatrix},
\qquad
R_1=\begin{pmatrix}0&I\\I&0\end{pmatrix},
\qquad
R_2=\begin{pmatrix}0&B^{\mathsf T}\\B&0\end{pmatrix},\\[1mm]
R_3=\begin{pmatrix}0&B\\B^{\mathsf T}&0\end{pmatrix},
\qquad
R_4=\begin{pmatrix}B&0\\0&B^{\mathsf T}\end{pmatrix},
\qquad
R_5=\begin{pmatrix}B^{\mathsf T}&0\\0&B\end{pmatrix}.
\end{gathered}
\]
They are pairwise disjoint and satisfy
 $R_0+R_1+R_2+R_3+R_4+R_5=J_{2n}$.

\begin{proposition}
\label{prop:rank}
The matrices $R_0,R_1,\ldots,R_5$ are the adjacency matrices of a
homogeneous noncommutative association scheme $\mathcal X(T)$ of rank six.
\end{proposition}

\begin{proof}
The matrices form a partition of $J_{2n}$ with $R_0=I_{2n}$.  They are closed under transpose:
$ R_i^{\mathsf T}=R_i\quad(0\leq i\leq3)$,
$R_4^{\mathsf T}=R_5$.
Every product $R_iR_j$ is a block matrix whose nonzero blocks belong to
$\operatorname{span}_{\mathbb C}\{I,J,B,B^{\mathsf T}\}$.
It is a routine-job to check that this space
is closed under multiplication.  The position of the nonzero blocks then
shows that every $R_iR_j$ belongs to
$\operatorname{span}_{\mathbb C}\{R_0,\ldots,R_5\}$.  Hence the six
relations form a homogeneous coherent configuration.
The adjacency algebra is not commutative.  Indeed,
$R_1R_4\neq R_4R_1$.
\end{proof}

The adjacency matrix of $\Gamma(T)$ is the merging $A=R_2+R_4$.
Thus Proposition~\ref{prop:rank} gives a rank-six association
scheme containing the graph.  We next prove that no proper fusion of this
scheme contains $A$.

\begin{theorem}
\label{thm:wl}
For every doubly regular tournament $T$, the Weisfeiler--Leman closure of
$\Gamma(T)$ is precisely $\mathcal X(T)$.  In particular,
$ \operatorname{rank}\operatorname{WL}(\Gamma(T))=6$.
\end{theorem}

\begin{proof}
Since $A\in\mathbb C\mathcal X(T)$, the WL-closure of $\Gamma(T)$ is a fusion
of $\mathcal X(T)$.  Conversely, we recover all six basis matrices using
operations under which every coherent algebra is closed.  First,
$A\circ A^{\mathsf T}=R_2$.
By the defining relations it follows that
$R_4=A-R_2$, $R_5=A^{\mathsf T}-R_2$.
Furthermore, 
$R_2R_4=qR_2+(q+1)R_3$,
and hence
$ R_3=\frac{1}{q+1}(R_2R_4-qR_2)$.
Finally, $R_0=I_{2n}$ and
$R_1=J_{2n}-R_0-R_2-R_3-R_4-R_5$.
Thus every basis matrix of $\mathcal X(T)$ belongs to the coherent algebra
generated by $A$, proving the assertion.
\end{proof}

By Proposition~\ref{ranksix}, rank six is the smallest possible
rank of a noncommutative association scheme in which a proper DSRG can occur
as a nontrivial merging.  The construction therefore attains the general
lower bound for every choice of the initial doubly regular tournament.

\subsection{Automorphisms and intransitivity}

The rank-six closure also makes it possible to relate the automorphisms of
$\Gamma(T)$ to those of $T$.  The relation $R_1$ is a perfect matching whose
classes are
\begin{equation}
 \{(x,0),(x,1)\},\qquad x\in V(T).
 \label{pairs}
\end{equation}
Since every automorphism of a graph preserves every basis relation of its
WL-closure, $\Aut(\Gamma(T))$ preserves this system of matched pairs.  Its
induced action on the set of pairs may preserve the two components of $R_4$,
or interchange them.  In the first case it induces an automorphism of $T$;
in the second it induces an isomorphism $T\to T^{\mathsf T}$.  Consequently,
the induced permutation group is a subgroup of $\Aut^{\pm}(T)$.

We first record a simple group-theoretic observation.

\begin{lemma}
\label{ext}
Let $T$ be a doubly regular tournament.  If $\Aut^{\pm}(T)$ is transitive on $V(T)$, then
$\Aut(T)$ is transitive on $V(T)$.
\end{lemma}

\begin{proof}
Put $ G=\Aut^{\pm}(T)$, $ H=\Aut(T).$
Then $H\trianglelefteq G$ and $[G:H]\leq2$.  Since $H$ is normal, the group
$G$ permutes the $H$-orbits transitively; hence they all have the same size.
Moreover, the number of $H$-orbits divides $[G:H]$, so it is either one or
two.  Every doubly regular tournament has odd order, and therefore its vertex set cannot be
the union of two equal-sized $H$-orbits.  Thus $H$ has one orbit and is
transitive.
\end{proof}

\begin{theorem}
\label{intrans}
If a doubly regular tournament $T$ is not vertex-transitive, then
$\Gamma(T)$ is not vertex-transitive.
\end{theorem}

\begin{proof}
Suppose that $\Aut(\Gamma(T))$ were transitive.  Since it preserves the
matched pairs in \eqref{pairs}, its induced group on these
pairs would also be transitive.  As observed above, this induced group is a
subgroup of $\Aut^{\pm}(T)$.  Therefore $\Aut^{\pm}(T)$ would be transitive,
and Lemma~\ref{ext} would imply that $\Aut(T)$ is
transitive, a contradiction.
\end{proof}

\begin{corollary}
\label{nsc}
If $T$ is not vertex-transitive, then $\mathcal X(T)$ is a non-Schurian
association scheme.
\end{corollary}

\begin{proof}
By Theorem~\ref{intrans}, the group
$\Aut(\Gamma(T))$ is intransitive.  On the other hand,
$\mathcal X(T)=\operatorname{WL}(\Gamma(T))$ is homogeneous and
$ \Aut(\mathcal X(T))=\Aut(\Gamma(T))$.
If $\mathcal X(T)$ were Schurian, its automorphism group would contain a
transitive group whose orbitals are its basis relations, contradicting the
intransitivity of $\Aut(\mathcal X(T))$.
\end{proof}

\subsection{The Szekeres tournaments}

We now recall the doubly regular tournaments of Szekeres type
\cite{ito,sz}.  Let $p\equiv5\pmod8$ be a prime, $p\geq13$,
and let $H_4$ be the subgroup of fourth powers in
$\mathbb Z_p^{\times}$.  Its four cosets are denoted by
$H_4$, $2H_4$, $4H_4$, $8H_4$.
For these primes, $-H_4=4H_4$, $-2H_4=8H_4$.

The \emph{Szekeres tournament} $S_p$ has vertex set
$V(S_p)=\bigl(\mathbb Z_p\times\{0,1\}\bigr)\sqcup\{\infty\}$. 
Thus, \((a,0)\) and \((a,1)\) denote the vertices corresponding to
\(a\in\mathbb Z_p\) in the first and second copies, respectively, and 
$\infty$ is an additional symbol.
The out-neighbourhoods are defined by the following
\begin{align*}
N^+_{S_p}\bigl((a,0)\bigr) &= \bigl(a+(H_4\cup2H_4)\bigr)\times\{0\}
 \;\cup\;\{(a,1)\} \;\cup\; \bigl(a+(H_4\cup8H_4)\bigr)\times\{1\}, \\
N^+_{S_p}\bigl((a,1)\bigr) &= \bigl(a+(H_4\cup8H_4)\bigr)\times\{0\}
 \;\cup\;\{\infty\} \;\cup\; \bigl(a+(4H_4\cup8H_4)\bigr)\times\{1\}, \\
N^+_{S_p}(\infty)  &=\mathbb Z_p\times\{0\}.
\end{align*}
Szekeres \cite{sz} proved that $S_p$ is a doubly regular tournament of order $2p+1$,
out-valency $p$, and common-out-neighbour number $(p-1)/2$.  
Translations act simultaneously on its two field copies 
$(a,i)\longmapsto(a+c,i)$, where  $i\in\{0,1\}$ and $c\in\mathbb Z_p$;
and fix $\infty$. 
The full automorphism group was determined by Ito \cite{ito}.  
In particular, the following fact is essential here.

\begin{proposition}[\cite{ito}]
\label{prop:Szekeres-intransitive}
For every prime $p\equiv5\pmod8$, $p\geq13$, the Szekeres tournament $S_p$
is not vertex-transitive.
\end{proposition}

\subsection{The main result}

Let $B_p$ be the adjacency matrix of $S_p$, and define
$A_p=\left(\begin{smallmatrix}
B_p & B_p^{\mathsf T}\\
B_p & B_p^{\mathsf T}
\end{smallmatrix}\right).$
We denote the corresponding digraph by $\Gamma_p$.

\begin{theorem}
\label{thm:main-infinite-family}
For every prime $p\equiv5\pmod8$, $p\geq13$, the graph $\Gamma_p$ is an
intransitive directed strongly regular graph with parameters
$(v,k,t,\lambda,\mu)=(4p+2,2p,p,p-1,p)$.
Its Weisfeiler--Leman closure is a homogeneous noncommutative association
scheme of rank six, and this scheme is non-Schurian.  Consequently, the
graphs $\Gamma_p$ form an infinite family of intransitive DSRGs with
homogeneous WL-closure of the smallest possible noncommutative rank.
\end{theorem}

\begin{proof}
The tournament $S_p$ has order $2p+1$ and out-valency $p$.  Applying
Lemma~\ref{thm:doubling} with $m=p$ gives the stated  parameters.  By
Theorem~\ref{thm:wl}, the WL-closure of $\Gamma_p$ is the
homogeneous noncommutative rank-six scheme $\mathcal X(S_p)$.
Proposition~\ref{prop:Szekeres-intransitive} and
Theorem~\ref{intrans} show that $\Gamma_p$ is
intransitive.  Its WL-closure is therefore non-Schurian by
Corollary~\ref{nsc}.  Finally, \emph{Dirichlet's theorem} on
primes in arithmetic progressions guarantees that there are infinitely many
primes congruent to $5$ modulo $8$.
\end{proof}

The first members of this infinite family have the following parameters:
(54,26,13,12,13), (118,58,29,28,29), (150,74,37,36,37) for 
$p=13$, 29, 37, respectively.  

\subsection{Rigid examples}
\label{rigid}

The analysis of the smallest known intransitive DSRGs whose
WL-closure is a homogeneous association scheme of rank six,
namely the examples of orders $30$ and $38$, revealed that both arise from
the doubling construction of Section~\ref{sec:main}.  The
underlying doubly regular tournaments have orders $15$ and $19$,
respectively.

A brief search through the catalogue of doubly regular tournaments produced
further, sporadic examples.  The smallest rigid examples found in this way
arise from doubly regular tournaments of order $23$.
The doubling construction gives
a DSRG with parameters $(46,22,11,10,11)$.
The complete catalogue~(see \cite{mc}) contains $37$ pairwise nonisomorphic doubly regular
tournaments of order $23$.  In the ordering used in
McKay's data file, representatives $2$ and $37$ have trivial automorphism
group and are not isomorphic to their reverses.  Consequently, their
extended automorphism groups are trivial.  Since every automorphism of the
doubled graph induces an automorphism or an anti-automorphism of the
underlying tournament, the corresponding DSRGs are rigid.  Nevertheless,
their WL-closure is a homogeneous noncommutative association
schemes of rank six. These examples show that group-theoretic symmetry,
combinatorial regularity, and WL-regularity are genuinely independent aspects
of directed strongly regular graphs.

\section*{Acknowledgements}
The author acknowledges support from the Slovak Scientific Grant Agency VEGA under Grants No.~1/0069/23 and No.~1/0011/25, and from the Slovak Research and Development Agency under Contracts No.~APVV-22-0005 and No.~APVV-23-0076.
The author would also like to thank Professor Mikhail Klin for many years of fruitful collaboration in Algebraic Graph Theory.


\end{document}